\documentclass{amsart}
\usepackage{amsfonts, amsbsy, amsmath, amssymb}

\newtheorem{thm}{Theorem}[section]
\newtheorem{lem}[thm]{Lemma}
\newtheorem{cor}[thm]{Corollary}
\newtheorem{prop}[thm]{Proposition}
\newtheorem{exmp}[thm]{Example}

\newtheorem{rmk}[thm]{Remark}

\newtheorem{ques}[thm]{Question}
\newtheorem{thm-con}[thm]{Theorem-Conjecture}
\numberwithin{equation}{section}

\theoremstyle{definition}

\newcommand{\x}{{\tt x}}
\newcommand{\T}{{\tt t}}

\begin{document}

\title{On Global $\mathcal P$-Forms}

\author[Xiang-dong Hou]{Xiang-dong Hou}
\address{Department of Mathematics and Statistics,
University of South Florida, Tampa, FL 33620}
\email{xhou@usf.edu}
\thanks{* Research partially supported by NSA Grant H98230-12-1-0245.}



\begin{abstract}
Let $\Bbb F_q$ be a finite field with $\text{char}\,\Bbb F_q=p$ and $n>0$ an integer with $\text{gcd}(n, \log_pq)=1$. Let $(\ )^*:\Bbb F_q(\x_0,\dots,\x_{n-1})\to\Bbb F_q(\x_0,\dots,\x_{n-1})$ be the $\Bbb F_q$-monomorphism defined by $\x_i^*=\x_{i+1}$ for $0\le i< n-1$ and $\x_{n-1}^*=\x_0^q$. For $f,g\in\Bbb F_q(\x_0,\dots,\x_{n-1})\setminus\Bbb F_q$, define $f\circ g=f(g,g^*,\dots,g^{(n-1)*})$. Then $(\Bbb F_q(\x_0,\dots,\x_{n-1})\setminus\Bbb F_q,\,\circ)$ is a monoid whose invertible elements are called global $\mathcal P$-forms. Global $\mathcal P$-forms were first introduced by H. Dobbertin in 2001 with $q=2$ to study certain type of permutation polynomials of $\Bbb F_{2^m}$ with $\text{gcd}(m,n)=1$; global $\mathcal P$-forms with $q=p$ for an arbitrary prime $p$ were considered by W. More in 2005. In this paper, we discuss some fundamental questions about global $\mathcal P$-forms, some of which are answered and others remain open. 
\end{abstract}

\maketitle

\section{introduction}

Let $F$ be a field and $F(\x_0,\dots,\x_{n-1})$ the field of rational functions in $\x_0,\dots,\x_{n-1}$ over $F$. Let $u>0$ be an integer such that $x^n-u$ is irreducible over $\Bbb Q$, which is equivalent to saying that for every prime divisor $r$ of $n$, $u$ is not an $r$th power of an integer \cite[Theorem~8.1.6]{Kar89}. Let 
\[
(\ )^*:\ F(\x_0,\dots,\x_{n-1})\longrightarrow F(\x_0,\dots,\x_{n-1})
\]
be the $F$-monomorphism defined by $\x_i^*=\x_{i+1}$ for $0\le i< n-1$ and $\x_{n-1}^*=\x_0^u$. For $f\in F(\x_0,\dots,\x_{n-1})$ and $i\ge 0$, we write
\[
f^{\overbrace{*\cdots*}^i}=f^{i*}.
\]
For $f,g\in F(\x_0,\dots,\x_{n-1})$, we define
\begin{equation}\label{1.1}
f\circ g=f(g,g^*,\dots,g^{(n-1)*})
\end{equation}
whenever the right side is meaningful. Writing $f=\frac{f_1}{f_2}$, where $f_1,f_2\in F[\x_0,\dots,\x_{n-1}]$ with $f_2\ne 0$, we have 
\[
f\circ g=\frac{f_1(g,g^*,\dots,g^{(n-1)*})}{f_2(g,g^*,\dots,g^{(n-1)*})}.
\]
Hence for $f\circ g$ to be defined for all $f\in F(\x_0,\dots,\x_{n-1})$, it is necessary and sufficient that $g,g^*,\dots,g^{(n-1)*}$ be algebraically independent over $F$. We will see in Section~3 that $g,g^*,\dots,g^{(n-1)*}$ are algebraically independent over $F$ if and only if $g$ is not a constant.

We are primarily interested in the case $F=\Bbb F_q$ with $\text{char}\,\Bbb F_q=p$, $u=q$, and $\text{gcd}(n,\log_pq)=1$. In this case, we also see that $(\Bbb F_q(\x_0,\dots,\x_{n-1})\setminus\Bbb F_q,\,\circ)$ forms a monoid with identity $\x_0$. The invertible elements of this monoid are called {\em global $\mathcal P$-forms} (in $\x_0,\dots,\x_{n-1}$ over $\Bbb F_q$) and the group they form is denoted by $\mathcal G(n,q)$.

Global $\mathcal P$-forms were first introduced by Dobbertin in \cite{Dob02} with $q=2$ and were later generalized by More \cite{Mor06} to the case $q=p$ for an arbitrary prime $p$. The motivation of this notion, according to \cite{Dob02}, is in the study of certain permutation polynomials of finite fields that possess a ``uniform representation''. More precisely, let $n,n',m$ be positive integers such that $nn'\equiv 1\pmod m$ and let $f\in\mathcal G(n,q)$. Denote the inverse of $f$ in $\mathcal G(n,q)$ by $f^{(-1)}$ and put $\widetilde f=f(\x,\x^{q^{n'}},\dots,\x^{q^{(n-1)n'}})\in\Bbb F_q(\x)$. (Here we need to assume that the denominator of $f(\x,\x^{q^{n'}},\dots,\x^{q^{(n-1)n'}})$ is not $0$. By choosing $n'$ large enough subject to the condition $nn'\equiv 1\pmod m$, this requirement is satisfied.) Let $D$ denote the set of all $x\in\Bbb F_{q^m}$ at which the rational functions $\widetilde f$ and $\widetilde{f^{(-1)}}\circ \widetilde f\in\Bbb F_q(\x)$ are both defined. Then for all $x\in D$ we have 
\[
\begin{split}
\widetilde{f^{(-1)}}\bigl(\widetilde f(x)\bigr)\,
&=f^{(-1)}\bigl(\widetilde f(x),\widetilde f(x)^{q^{n'}},\dots,\widetilde f(x)^{q^{(n-1)n'}}\bigr)\cr
&=f^{(-1)}\bigl(f(x,x^{q^{n'}},\dots,x^{q^{(n-1)n'}}),f(x^{q^{n'}},x^{q^{2n'}},\dots,x^{q^{(n-1)n'}},x^q),\dots,\cr
&\kern 1.45cm f(x^{q^{(n-1)n'}},x^q,x^{q^{1+n'}},\dots,x^{q^{1+(n-2)n'}})\bigr)\cr
&=f^{(-1)}\bigl(f(x,x^{q^{n'}},\dots,x^{q^{(n-1)n'}}),f^*(x,x^{q^{n'}},\dots,x^{q^{(n-1)n'}}),\dots,\cr
&\kern 1.45cm f^{(n-1)*}(x,x^{q^{n'}},\dots,x^{q^{(n-1)n'}})\bigr)\cr
&=(f^{(-1)}\circ f)(x,x^{q^{n'}},\dots,x^{q^{(n-1)n'}})\cr
&=x.
\end{split}
\]
In particular, $\widetilde f$ is one-to-one on $D$. Write $\widetilde f=\frac{f_1}{f_2}$ where $f_1,f_2\in\Bbb F_q[\x]$, $f_2\ne 0$, $\text{gcd}(f_1,f_2)=1$, and put $g=f_1f_2^{q^m-2}\in\Bbb F_q[\x]$. Then $g$ is one-to-one on $D$. If, in addition, one can show that $g$ is also one-to-one on $\Bbb F_{q^m}\setminus D$ and $g(\Bbb F_{q^m}\setminus D)\cap g(D)=\emptyset$, it follows that $g$ is a permutation polynomial of $\Bbb F_{q^m}$. In \cite{Dob99}, Dobbertin found a formula for the inverse of the Kasami permutation polynomial. The proof in \cite{Dob99} implicitly relied on an extraordinary global $\mathcal P$-form $Q_n\in\mathcal G(n,2)$ and its inverse, both of which were made explicit later in \cite{Dob02}. (We will revisit the global $\mathcal P$-form $Q_n$ in detail shortly.)

It is clear that $\mathcal G(1,q)\cong \text{PGL}(2,q)$. For $m\mid n$, there is a natural embedding $\mathcal G(m,q)\hookrightarrow \mathcal G(n,q)$. There is another embedding 
\[
\begin{array}{ccccc}
\phi_n:&\Bbb Z[q^{\frac 1n}]^\times&\longrightarrow&\mathcal G(n,q)\cr
&a_0+a_1q^{\frac 1n}+\cdots+a_{n-1}q^{\frac{n-1}n}&\longmapsto&\x_0^{a_0}\cdots\x_{n-1}^{a_{n-1}},& a_0,\dots,a_n\in\Bbb Z.
\end{array}
\]
The image of $\phi_n$ consists of all global $\mathcal P$-forms that are rational monomials. A fundamental question is whether $\mathcal G(n,q)$ is generated by $\mathcal G(1,q)$ and $\text{Im}\,\phi_n$. The answer is not known.

In \cite{Dob02}, Dobbertin proved that 
\begin{equation}\label{1.2}
Q_n:=\frac 1{\x_0\x_1}\Bigl(\x_0^2+\sum_{i=1}^{n-1}\x_i+n+1\Bigr)\in\mathcal G(n,2).
\end{equation}
Moreover,
\begin{equation}\label{1.3}
Q_n^{(-1)}=\sum_{i=1}^n\sum_{e_0,\dots,e_{i-1}}\x_0^{e_0}\x_1^{e_1}\cdots\x_{i-1}^{e_{i-1}},
\end{equation}
where $e_0,\dots,e_{i-1}\in\{\pm1\}$ are subject to the conditions $e_0=-1$ (when $i<n$), $e_0=\pm 1$ (when $i=n$), $e_{i-1}=-1$, and $(e_{j-1},e_j)\ne(1,1)$ for all $0<j<i$. It is not known if $Q_n\in\langle\mathcal G(1,n)\cup\text{Im}\,\phi_n\rangle$. It follows from \eqref{1.2} and \eqref{1.3} (or by direct computation) that $o(Q_2)=2$. Dobbertin \cite{Dob02} posed the question whether $o(Q_n)$ is infinite for $n>2$. We will prove that $o(Q_n)=\infty$ for $n>2$.

The paper is organized as follows. In Section 2 we assume that $F$ is an arbitrary field and $x^n-u\in\Bbb Z[\x]$ is irreducible over $\Bbb Q$, where $n$ and $u$ are positive integers. We introduce the notion of a $\Bbb Z[u^{\frac 1n}]$-valued degree for functions in $F(\x_0,\dots,\x_{n-1})$, which allows us to prove that if $g\in F(\x_0,\dots,\x_{n-1})\setminus F$, then $g,g^*,\dots,g^{(n-1)*}$ are algebraically independent over $F$. In Section 3, we assume that $F=\Bbb F_q$, $u=q$, and $\text{gcd}(n,\log_pq)=1$, where $p=\text{char}\,\Bbb F_q$. After a discussion of the basic properties of global $\mathcal P$-forms, we prove that $o(Q_n)=\infty$ for $n>3$. The proof is based on the computation of the $\Bbb Z[2^{\frac 1n}]$-valued degree of $Q_n$. Section 4 is devoted the structure of the group $\mathcal G(n,q)$ of global $\mathcal P$-forms. Several embeddings are described:  $\mathcal G(m,q)\hookrightarrow\mathcal G(n,q)$, where $m\mid n$; $\Bbb Z[q^{\frac 1n}]^\times\hookrightarrow\mathcal G(n,q)$; $\mathcal G(n,q)\hookrightarrow\text{Cr}_n(\Bbb F_q)$, where $\text{Cr}_n(\Bbb F_q)=\text{Aut}(\Bbb F_q(\x_0,\dots,\x_{n-1})/\Bbb F_q)$ is the Cremona group of $\Bbb F_q$ in $n$ dimensions. There are two degree functions $d_{\max},\ d_{\min}:\mathcal G(n,q)\to\Bbb Z[q^{\frac 1n}]^\times\cup\{0\}$, and $\mathcal H(n,q):=\{f\in\mathcal G(n,q):d_{\max}(f)d_{\min}(f)>0\}$ is a subgroup of $\mathcal G(n,q)$ of finite index. Left coset representatives of $\mathcal H(n,q)$ in $\mathcal G(n,q)$ are determined and so is the structure of the quotient group $\mathcal H(n,q)/\{f\in\mathcal H(n,q):d_{\max}(f)=d_{\min}(f)=1\}$. Section 4 also contains several open questions that are fundamental for a better understanding of the group $\mathcal G(n,q)$.  


\section{The $\Bbb Z[u^{\frac 1n}]$-Valued Degree}

Let $F$ be a field and let $n$ and $u$ be positive integers such that $\x^n-u$ is irreducible over $\Bbb Q$. For $0\ne f=\sum_{e_0,\dots,e_{n-1}\ge 0}c_{e_0,\dots,e_{n-1}}\x_0^{e_0}\cdots\x_{n-1}^{e_{n-1}}\in F[\x_0,\dots,\x_{n-1}]$, define
\begin{equation}\label{2.1}
d(f)=\max\{e_0+e_1u^{\frac 1n}+\cdots+e_{n-1}u^{\frac{n-1}n}:c_{e_0,\dots,e_{n-1}}\ne 0\}.
\end{equation}
We also define $d(0)=-\infty$. It is obvious that
\[
d(fg)=d(f)+d(g)\quad\text{for all}\ f,g\in F[\x_0,\dots,\x_{n-1}].
\]
The next result is more interesting.

\begin{lem}\label{L2.1}
Let $f,g\in F[\x_0,\dots,\x_{n-1}]$ be such that $d(g)>0$. Then 
\begin{equation}\label{2.2}
d(f\circ g)=d(f)d(g),
\end{equation}
where $f\circ g$ is defined in \eqref{1.1}.
\end{lem}

\begin{proof}It is clear that $d(g^*)=u^{\frac 1n}d(g)$. Thus in general we have $d(g^{i*})=u^{\frac in}d(g)$ for $i\ge 0$. Let $c\,\x_0^{e_0}\cdots\x_{n-1}^{e_{n-1}}$ ($c\in F\setminus\{0\}$) be the leading term of $f$ with respect to $d$. Then the leading term of $f\circ g$ equals that of $cg^{e_0}(g^*)^{e_1}\cdots(g^{(n-1)*})^{e_{n-1}}$. Note that 
\[
\begin{split}
d\bigl(g^{e_0}(g^*)^{e_1}\cdots(g^{(n-1)*})^{e_{n-1}}\bigr)\,&=e_0d(g)+e_1u^{\frac 1n}d(g)+\cdots+e_{n-1}u^{\frac{n-1}n}d(g)\cr
&=(e_0+e_1u^{\frac 1n}+\cdots+e_{n-1}u^{\frac{n-1}n})d(g)\cr
&=d(f)d(g).
\end{split}
\]
\end{proof}

For $f=\frac{f_1}{f_2}\in F(\x_0,\dots,\x_{n-1})$, where $f_1,f_2\in F[\x_0,\dots,\x_{n-1}]$, $f_2\ne 0$, we define
\[
d(f)=d(f_1)-d(f_2).
\]
Then the function $d:F(\x_0,\dots,\x_{n-1})\to\Bbb Z[u^{\frac 1n}]\cup\{-\infty\}$ has the following properties. For $f,g\in F(\x_0,\dots,\x_{n-1})$,
\begin{itemize}
  \item [(i)] $d(fg)=d(f)+d(g)$;
  \item [(ii)] $d(f+g)\le \max(d(f),d(g))$ and if $d(f)\ne d(g)$, then $d(f+g)=\max(d(f),d(g))$.
\end{itemize} 
In fact, $-d$ is a valuation of $F(\x_0,\dots,\x_{n-1})$ with the value group $(\Bbb Z[u^{\frac 1n}],+,\le)$. Lemma~\ref{L2.1} is still valid for $g\in F(\x_0,\dots,\x_{n-1})$ with $d(g)>0$.

\begin{lem}\label{L2.2} Let $f\in F[\x_0,\dots,\x_{n-1}]$ and $g\in F(\x_0,\dots,\x_{n-1})$ be such that $d(g)>0$. Then 
\begin{equation}\label{2.3}
d(f\circ g)=d(f)d(g).
\end{equation}
\end{lem}

\begin{proof} 
Write $f=\sum_{e_0,\dots,e_{n-1}\ge 0}c_{e_0,\dots,e_{n-1}}\x_0^{e_0}\cdots\x_{n-1}^{e_{n-1}}$. Then
\begin{equation}\label{2.4}
f\circ g= \sum_{e_0,\dots,e_{n-1}\ge 0}c_{e_0,\dots,e_{n-1}}g^{e_0}(g^*)^{e_1}\cdots(g^{(n-1)*})^{e_{n-1}}.
\end{equation}
Note that
\begin{equation}\label{2.5}
d\bigl(g^{e_0}(g^*)^{e_1}\cdots(g^{(n-1)*})^{e_{n-1}}\bigr)=(e_0+e_1 u^{\frac 1n}+\cdots+e_{n-1}u^{\frac{n-1}n})d(g),
\end{equation}
which are distinct for different $(e_0,\dots,e_{n-1})$. Therefore,
\[
d(f\circ g)=\max_{\substack{e_0,\dots,e_{n-1}\ge 0\cr c_{e_0,\dots,e_{n-1}}\ne 0}}(e_0+e_1 u^{\frac 1n}+\cdots+e_{n-1}u^{\frac{n-1}n})d(g)=d(f)d(g).
\]
\end{proof}

\begin{thm}\label{T2.3}
Let $g\in F(\x_0,\dots,\x_{n-1})\setminus F$. Then $g,g^*,\dots,g^{(n-1)*}$ are algebraically independent over $F$.
\end{thm}

\begin{proof}
We may assume $d(g)>0$. (If $d(g)<0$, consider $\frac 1g$. If $d(g)=0$, then $g=a+g_1$, where $a\in F\setminus\{0\}$ and $g_1\in F(\x_0,\dots,\x_{n-1})$ with $d(g_1)<0$. We then consider $\frac 1{g_1}$.) Assume to the contrary that there exists $0\ne f\in F[\x_0,\dots,\x_{n-1}]$ such that
\[
0=f(g,g^*,\dots,g^{(n-1)*})=f\circ g.
\]
Then by Lemma~\ref{L2.2},
\[
-\infty=d(f\circ g)=d(f)d(g)>-\infty,
\]
which is a contradiction.
\end{proof}
 
\noindent{\bf Note.} Theorem~\ref{T2.3} with $F=\Bbb F_p$ and $u=p$ was stated in \cite{Mor06}. But unfortunately, the proof there is seriously flawed.

\medskip
As a consequence of Theorem~\ref{T2.3}, the operation (composition) $f\circ g$ in \eqref{1.1} is well defined for all $f\in F(\x_0,\dots,\x_{n-1})$ and $g\in F(\x_0,\dots,\x_{n-1})\setminus F$. 

Next we extend the notion of the $\Bbb Z[u^{\frac 1n}]$-valued degree $d$. For $0\ne f\in F[\x_0,\dots,\x_{n-1}]$, define
\[
\begin{split}
d_{\max}(f)\,&=d(f)=\text{the maximum $d$-value of the terms of $f$},\cr
d_{\min}(f)\,&=\text{the minimum $d$-value of the terms of $f$}.
\end{split}
\]
Also define $d_{\max}(0)=-\infty$ and $d_{\min}(0)=\infty$. For $f=\frac{f_1}{f_2}\in F(\x_0,\dots,\x_{n-1})$, where $f_1,f_2\in F[\x_0,\dots,\x_{n-1}]$, $f_2\ne 0$, define
\begin{align}
&d_{\max}(f)=d(f)=d_{\max}(f_1)-d_{\max}(f_2),\notag \\
&d_{\min}(f)=d_{\min}(f_1)-d_{\min}(f_2),\notag \\
&\delta(f)=(d_{\max}(f),d_{\min}(f)).\label{2.6}
\end{align}
Obviously, $\delta(f+g)=\delta(f)+\delta(g)$ for $f,g\in F(\x_0,\dots,\x_{n-1})$.

\begin{exmp}\label{E2.4}\rm 
Let $F=\Bbb F_2$, $u=2$ and 
\[
Q_n=\frac 1{\x_0\x_1}\Bigl(\x_0^2+\sum_{i=1}^{n-1}\x_i+n+1\Bigr)\in\Bbb F_2(\x_0,\dots,\x_{n-1})
\]
as in \eqref{1.2}. We have
\begin{equation}\label{2.7}
d_{\max}(Q_n)=1-2^{\frac 1n},
\end{equation}
\begin{equation}\label{2.8}
d_{\min}(Q_n)=
\begin{cases}
-1&\text{if $n$ is odd},\cr
-1-2^{\frac 1n}&\text{if $n$ is even.}
\end{cases}
\end{equation}
\end{exmp}

\begin{lem}\label{L2.5}
Let $f\in F(\x_0,\dots,\x_{n-1})$ and $g\in F(\x_0,\dots,\x_{n-1})\setminus F$. Then 
\begin{equation}\label{2.9}
d_{\max}(f\circ g)=
\begin{cases}
d_{\max}(f)d_{\max}(g)&\text{if}\ d_{\max}(g)>0,\cr
d_{\min}(f)d_{\max}(g)&\text{if}\ d_{\max}(g)<0;
\end{cases}
\end{equation}
\begin{equation}\label{2.10}
d_{\min}(f\circ g)=
\begin{cases}
d_{\min}(f)d_{\min}(g)&\text{if}\ d_{\min}(g)>0,\cr
d_{\max}(f)d_{\min}(g)&\text{if}\ d_{\min}(g)<0.
\end{cases}
\end{equation}
\end{lem}

\begin{proof}
The proof is similar to that of Lemma~\ref{L2.2}. We only prove one case: $d_{\min}(f\circ g)=d_{\max}(f)d_{\min}(g)$ when $d_{\min}(g)<0$. It suffices to prove the formula with $f\in F[\x_0,\dots,\x_{n-1}]$, say $f=\sum_{e_0,\dots,e_{n-1}\ge 0}c_{e_0,\dots,e_{n-1}}\x_0^{e_0}\cdots\x_{n-1}^{e_{n-1}}$. Then $f\circ g$ is given by \eqref{2.4}. Note that \eqref{2.5} still holds with $d$ replaced by $d_{\min}$. It is also easy to see that if $\alpha,\beta\in F(\x_0,\dots,\x_{n-1})$ are such that $d_{\min}(\alpha)\ne d_{\min}(\beta)$, then $d_{\min}(\alpha+\beta)=\min(d_{\min}(\alpha),d_{\min}(\beta))$. Therefore we have
\[
\begin{split}
d_{\min}(f\circ g)\,&=\min_{\substack{e_0,\dots,e_{n-1}\ge 0\cr c_{e_0,\dots,e_{n-1}}\ne 0}}(e_0+e_1u^{\frac 1n}+\cdots+e_{n-1}u^{\frac{n-1}n})d_{\min}(g)\cr
&=d_{\min}(g)\max_{\substack{e_0,\dots,e_{n-1}\ge 0\cr c_{e_0,\dots,e_{n-1}}\ne 0}}(e_0+e_1u^{\frac 1n}+\cdots+e_{n-1}u^{\frac{n-1}n})\cr
&=d_{\max}(f)d_{\min}(g).
\end{split}
\]
\end{proof}

\begin{rmk}\label{R2.6}\rm
In Lemma~\ref{L2.5}, if $d_{\max}(g)=0$, then $g=a+g_1$ for some $a\in F\setminus\{0\}$ and $g_1\in F(\x_0,\dots,\x_{n-1})$ with $d_{\max}(g_1)<0$. Hence
\[
d_{\max}(f\circ g)=d_{\max}\bigl(f(\x_0+a,\dots,\x_{n-1}+a)\circ g_1\bigr)=d_{\min}\bigl(f(\x_0+a,\dots,\x_{n-1}+a)\bigr)d_{\max}(g_1).
\]
Similarly, if $d_{\min}(g)=0$, then $g=b+g_2$ for some $b\in F\setminus\{0\}$ and $g_2\in F(\x_0,\dots,\x_{n-1})$ with $d_{\min}(g)>0$. Hence
\[
d_{\min}(f\circ g)=d_{\min}\bigl(f(\x_0+b,\dots,\x_{n-1}+b)\bigr)d_{\min}(g_2).
\]
\end{rmk}

Let $\frak M(n,F)$ denote the set of all rational monomials in $F(\x_0,\dots,\x_{n-1})$, that is,
\[
\frak M(n,F)=\{\x_0^{e_0}\cdots\x_{n-1}^{e_{n-1}}:e_0,\dots,e_{n-1}\in\Bbb Z\}.
\]
(Note. Rational monomials, according to our definition, have coefficient $1$.)
It is easy to verify that $(\frak M(n,F),\,\cdot,\,\circ)$ is a commutative ring whose ``addition'' is the ordinary multiplication $\cdot$ and whose ``multiplication'' is the composition $\circ$ defined in \eqref{1.1}. The additive and multiplicative identities are $1$ and $\x_0$, respectively. Moreover, the mapping
\begin{equation}\label{2.11}
\begin{array}{ccc}
\frak M(n,F)&\longrightarrow&\Bbb Z[u^{\frac 1n}]\cr
f&\longmapsto&d(f)
\end{array}
\end{equation}
is a ring isomorphism.


\section{Global $\mathcal P$-Forms}

From now on we assume, in the notation of Section 2, that $F=\Bbb F_q$, $u=q$, and $\text{gcd}(n,\log_pq)=1$, where $p=\text{char}\,\Bbb F_q$. Clearly, $\Bbb F_q(\x_0,\dots,\x_{n-1})\setminus\Bbb F_q$ is closed under $\circ$ and $(\ )^*$. For $f,g\in \Bbb F_q(\x_0,\dots,\x_{n-1})\setminus \Bbb F_q$, we have $(f\circ g)^*=f\circ g^*=f^*\circ g$; the first equal sign follows from the definition of $\circ$ and $(\ )^*$; the second equal sign relies on the fact that $g^{n*}=g^q$. Therefore for $f,g,h\in\Bbb F_q(\x_0,\dots,\x_{n-1})\setminus\Bbb F_q$, we have 
\[
\begin{split}
&(f\circ g)\circ h\cr
=\,&(f\circ g)(h,h^*,\dots,h^{(n-1)*})\cr
=\,&f\bigl(g(h,h^*,\dots,h^{(n-1)*}),g^*(h,h^*,\dots,h^{(n-1)*}),\dots,g^{(n-1)*}(h,h^*,\dots,h^{(n-1)*})\bigr)\cr
=\,&f(g\circ h,g^*\circ h,\dots,g^{(n-1)*}\circ h)\cr
=\,&g\bigl(g\circ h,(g\circ h)^*,\dots,(g\circ h)^{(n-1)*}\bigr)\cr
=\,&f\circ(g\circ h).
\end{split}
\]
It is obvious that $f\circ \x_0=\x_0\circ f$ for all $f\in\Bbb F_q(\x_0,\dots,\x_{n-1})\setminus\Bbb F_q$. Thus $(\Bbb F_q(\x_0,\dots,$ $\x_{n-1})\setminus\Bbb F_q,\,\circ)$ is a monoid with identity $\x_0$. The $i$th power and the inverse, if exists, of an element $f$ of this monoid are denoted by $f^{(i)}$ and $f^{(-1)}$, respectively. The invertible elements of $(\Bbb F_q(\x_0,\dots,\x_{n-1})\setminus\Bbb F_q,\,\circ)$, called {\em global $\mathcal P$-forms}, form the group $\mathcal G(n,q)$. 

\begin{prop}\label{P3.1}
Let $f\in\Bbb F_q(\x_0,\dots,\x_{n-1})\setminus\Bbb F_q$. The following statements are equivalent.
\begin{itemize}
  \item [(i)] $f\in\mathcal G(n,q)$.
  \item [(ii)] $\Bbb F_q(f,f^*,\dots,f^{(n-1)*})=\Bbb F_q(\x_0,\dots,\x_{n-1})$.
  \item [(iii)] $f$ has a left inverse (with respect to $\circ$).
  \item [(iv)] $f$ has a right inverse (with respect to $\circ$).
\end{itemize} 
\end{prop}

\begin{proof}
(i) $\Rightarrow$ (ii). For all $0\le i\le n-1$, we have
\[
\x_i=\x_0^{i*}=(f^{(-1)}\circ f)^{i*}=(f^{(-1)})^{i*}\circ f\in\Bbb F_q(f,f^*,\dots,f^{(n-1)*}).
\]
Hence $\Bbb F_q(\x_0,\dots,\x_{n-1})=\Bbb F_q(f,f^*,\dots,f^{(n-1)*})$.

\medskip
(ii) $\Rightarrow$ (iii). Since $\x_0\in\Bbb F_q(f,f^*,\dots,f^{(n-1)*})$, there exists $g\in\Bbb F_q(\x_0,\dots,\x_{n-1})$ such that $g(f,f^*,\dots,f^{(n-1)*})=\x_0$, that is, $g\circ f=\x_0$.

\medskip
(iii) $\Rightarrow$ (iv). Let $g\in\Bbb F_q(\x_0,\dots,\x_{n-1})\setminus\Bbb F_q$ be such that $g\circ f=\x_0$. Then $(f\circ g)\circ f=f$, i.e., 
\[
(f\circ g)(f,f^*,\dots,f^{(n-1)*})=f.
\]
Since $f,f^*,\dots,f^{(n-1)*}$ are algebraically independent over $\Bbb F_q$, we have $f\circ g=\x_0$.

\medskip
(iv) $\Rightarrow$ (i). Let $g\in\Bbb F_q(\x_0,\dots,\x_{n-1})\setminus\Bbb F_q$ be such that $f\circ g=\x_0$. Then $(g\circ f)\circ g=g$, that is,
\[
(g\circ f)(g,g^*,\dots,g^{(n-1)*})=g.
\]
Since $g,g^*,\dots,g^{(n-1)*}$ are algebraically independent over $\Bbb F_q$, we have $g\circ f=\x_0$. Hence $f\in \mathcal G(n,q)$. 
\end{proof}

When $q=2$ and $n\ge 2$, Dobbertin found that $Q_n=\frac 1{\x_0\x_1}(\x_0^2+\sum_{i=1}^{n-1}\x_i+n+1)\in\mathcal G(n,2)$ with $Q_n^{(-1)}$ given by \eqref{1.3}. The verification of this claim is straightforward for $n=2$ but is quite complicated for general $n$; we refer the reader to \cite{Dob99} for the details. It is known that $o(Q_2)=2$, and an open question in \cite{Dob02} asks if $o(Q_n)$ is infinite for $n\ge 3$. We now answer this question affirmatively.

\begin{thm}\label{T3.2}
\begin{itemize}
  \item [(i)] For odd $n\ge 3$ and any $m\ge 0$,
\begin{equation}\label{3.1}
\begin{cases}
d_{\max}(Q_n^{(m)})=(-1)^m(-1+2^{\frac 1n})^{\lceil\frac m2\rceil},\cr
d_{\min}(Q_n^{(m)})=(-1)^m(-1+2^{\frac 1n})^{\lfloor\frac m2\rfloor}.
\end{cases}
\end{equation}

\item[(ii)] For even $n\ge 2$ and any $m\ge 0$,
\begin{equation}\label{3.2}
\begin{cases}
d_{\max}(Q_n^{(m)})=(-1)^m(-1+2^{\frac 2n})^{\lfloor\frac m2\rfloor}(-1+2^{\frac 1n})^{m-2\lfloor\frac m2\rfloor},\cr
d_{\min}(Q_n^{(m)})=(-1)^m(-1+2^{\frac 2n})^{\lfloor\frac m2\rfloor}(1+2^{\frac 1n})^{m-2\lfloor\frac m2\rfloor}.
\end{cases}
\end{equation}
\end{itemize}
\end{thm}

\begin{proof}
(i) By \eqref{2.7} and \eqref{2.8}, $d_{\max}(Q_n)=1-2^{\frac 1n}$ and $d_{\min}(Q_n)=-1$. Thus by Lemma~\ref{L2.5}, $d_{\max}(Q_n^{(2)})=d_{\min}(Q_n^{(2)})=-1+2^{\frac 1n}$. Repeated applications of Lemma~\ref{L2.5} give
\[
d_{\max}(Q_n^{(2k)})=d_{\min}(Q_n^{(2k)})=(-1+2^{\frac 1n})^k.
\]
Consequently,
\[
\begin{split}
d_{\max}(Q_n^{(2k+1)})\,&=(-1+2^{\frac 1n})^k(1-2^{\frac 1n}),\cr
d_{\max}(Q_n^{(2k+1)})\,&=(-1+2^{\frac 1n})^k(-1).
\end{split}
\]
Hence we have \eqref{3.1}.

\medskip
(ii) By \eqref{2.7} and \eqref{2.8}, $d_{\max}(Q_n)=1-2^{\frac 1n}$ and $d_{\min}(Q_n)=-1-2^{\frac 1n}$. By Lemma~\ref{L2.5}, $d_{\max}(Q_n^{(2)})=d_{\min}(Q_n^{(2)})=-1+2^\frac 2n$. Consequently,
\[
\begin{split}
d_{\max}(Q_n^{(2k)})\,&=d_{\min}(Q_n^{(2k)})=(-1+2^{\frac 2n})^k,\cr
d_{\max}(Q_n^{(2k+1)})\,&=(-1+2^{\frac 2n})^k(1-2^{\frac 1n}),\cr
d_{\min}(Q_n^{(2k+1)})\,&=(-1+2^{\frac 2n})^k(-1-2^{\frac 1n}). 
\end{split}
\]
Hence we have \eqref{3.2}.
\end{proof}

\begin{cor}\label{C3.3}
For $n>2$, we have $o(Q_n)=\infty$.
\end{cor}

\begin{proof} Since $n>2$, it follows from Theorem~\ref{T3.2} that $d_{\max}(Q_n^{(m)})\ne 1$ for all $m>0$. Thus $Q_n^{(m)}\ne \x_0$ for all $m>0$.
\end{proof}


\section{The Group $\mathcal G(n,q)$}

When $n=1$, the situation is quite simple. It is well known that $f\in\Bbb F_q(\x_0)$ generates $\Bbb F_q(\x_0)$ over $\Bbb F_q$ if and only if $f=\frac{a\x_0+b}{c\x_0+d}$, where $a,b,c,d\in\Bbb F_q$ and $ad-bc\ne 0$. Thus it follows from Proposition~\ref{P3.1} that
\begin{equation}\label{4.1}
\mathcal G(n,q)=\Bigl\{\,\frac{a\x_0+b}{c\x_0+d}:a,b,c,d\in\Bbb F_q,\ ad-bc\ne0\,\Bigr\}.
\end{equation}
For any field $F$, let $\text{Cr}_n(F)=\text{Aut}(F(\x_0,\dots,\x_{n-1})/F)$ be the {\em Cremona group} of $F$ in $n$ dimensions. Then we have $\mathcal G(1,q)\cong\text{PGL}(2,\Bbb F_q)\cong\text{Cr}_1(\Bbb F_q)$ with the obvious isomorphisms.

\begin{prop}\label{P4.1}
There is a group embedding 
\[
\begin{array}{cccc}
\gamma:&\mathcal G(n,q)&\longrightarrow&\text{\rm Cr}_n(\Bbb F_q)\cr
&f&\longmapsto&\gamma(f),
\end{array}
\]
where
\[
\begin{array}{cccc}
\gamma(f):&\Bbb F_q(\x_0,\dots,\x_{n-1})&\longrightarrow&  \Bbb F_q(\x_0,\dots,\x_{n-1})\cr
&h&\longmapsto&h\circ f^{(-1)}.
\end{array}
\]
\end{prop}

\begin{proof}
For $f\in\mathcal G(n,q)$, we have $f^{(-1)}\in\mathcal G(n,q)$, from which it follows that $\gamma(f)\in\text{Cr}_n(\Bbb F_q)$. For $f,g\in\mathcal G(n,q)$ and $h\in\Bbb F_q(\x_0,\dots,\x_{n-1})$, we have 
\[
\bigl(\gamma(f)\gamma(g)\bigr)(h)=\gamma(f)(h\circ g^{(-1)})=h\circ g^{(-1)}\circ f^{(-1)}=h\circ(f\circ g)^{(-1)}=\gamma(f\circ g)(h).
\]
So $\gamma(f\circ g)=\gamma(f)\gamma(g)$. That $\gamma$ is one-to-one is obvious.
\end{proof}

\noindent
{\bf Remark.} When $n=2$ and $F$ is an algebraically closed field, a set of generators of the Cremona group $\text{Cr}_2(F)$ is given by the Noether-Castelnuovo theorem \cite[Theorem~2.20]{KSC04} and a presentation of $\text{Cr}_2(F)$ is given in \cite{Bla12, Giz82}. When $n=2$ but $F$ is not algebraically closed or when $n\ge 3$, the situation is more difficult \cite{Pan99, Ser09}.

\begin{prop}\label{P4.2}
Assume $m\mid n$ and let $k=\frac nm$. Then the mapping 
\[
\begin{array}{cccc}
\psi_{m,n}:&\mathcal G(m,q)&\longrightarrow&\mathcal G(n,q)\cr
&f(\x_0,\dots,\x_{m-1})&\longmapsto&f(\x_0,\x_k,\dots,\x_{(m-1)k})
\end{array}
\]
is a group embedding. Moreover, $\text{\rm Im}\,\psi_{m,n}=\mathcal G(n,q)\cap\Bbb F_q(\x_0,\x_k,\dots,\x_{(m-1)k})$.
\end{prop}

\begin{proof} To prove the first claim, it suffices to show that $\psi_{m,n}(f\circ g)=\psi_{m,n}(f)\circ\psi_{m,n}(g)$ for all $f,g\in\Bbb F_q(\x_0,\dots,\x_{m-1})\setminus\Bbb F_q$. We have 
\[
\begin{split}
&\psi_{m,n}(f)\circ\psi_{m,n}(g)\cr
=\,&f(\x_0,\x_k,\dots,\x_{(m-1)k})\circ g(\x_0,\x_k,\dots,\x_{(m-1)k})\cr
=\,&f\bigl(g(\x_0,\x_k,\dots,\x_{(m-1)k}),g(\x_0,\x_k,\dots,\x_{(m-1)k})^{k*},\dots,g(\x_0,\x_k,\dots,\x_{(m-1)k})^{(m-1)k*}\bigr)\cr
=\,&f\bigl(g(\x_0,\x_k,\dots,\x_{(m-1)k}),g(\x_k,\dots,\x_{(m-1)k},\x_0^q),\dots,g(\x_{(m-1)k},\x_0^q,\dots,\x_{(m-2)k}^q)\bigr)\cr
=\,&(f\circ g)(\x_0,\x_k,\dots,\x_{(m-1)k})\cr
=\,&\psi_{m,n}(f\circ g).
\end{split}
\]

To prove that second claim, assume that $\alpha=f(\x_0,\x_k,\dots,\x_{(m-1)k})\in \mathcal G(n,q)$, where $f\in\Bbb F_q(\x_0,\dots,\x_{m-1})$. We show that $f(\x_0,\dots,\x_{m-1})\in\mathcal G(m,q)$. To this end, it suffices to show that $\alpha^{(-1)}\in\Bbb F_q(\x_0,\x_k,\dots,\x_{(m-1)k})$. We have 
\begin{equation}\label{4.2}
\alpha^{(-1)}(\alpha,\alpha^*,\dots,\alpha^{(m-1)*})=\x_0.
\end{equation}
We claim that $\alpha^*,\alpha^{(k+1)*},\dots,\alpha^{((m-1)k+1)*}$ are algebraically independent over\break
 $\Bbb F_q(\{\alpha^{i*}:0\le i\le n-1,\ i\not\equiv 1\pmod k\}\cup\{\x_0\})$. Otherwise, we have
\[
\begin{split}
n\,&=\text{tr.d}\,\Bbb F_q(\alpha,\alpha^*,\dots,\alpha^{(n-1)*})/\Bbb F_q \kern1cm\text{(tr.d = transcendence degree)}\cr
&<m+ \text{tr.d}\,\Bbb F_q(\{\alpha^{i*}:0\le i\le n-1,\ i\not\equiv 1\pmod k\}\cup\{\x_0\})/\Bbb F_q\cr
&\le m+\text{tr.d}\,\Bbb F_q(\{\x_i:0\le i\le n-1,\ i\not\equiv 1\pmod k\})/\Bbb F_q\cr
&=m+(n-m)=n,
\end{split}
\]
which is a contradiction. Therefore the left side of \eqref{4.2} does not involve $\alpha^*,\alpha^{(k+1)*},$ $\dots,\alpha^{((m-1)k+1)*}$. Since $\alpha,\alpha^*,\dots,\alpha^{(m-1)*}$ are algebraically independent over $\Bbb F_q$, this means that $\alpha^{(-1)}$ does not involve $\x_1,\x_{k+1},\dots,\x_{(m-1)k+1}$. In the same way, $\alpha^{(-1)}$ does not involve $\x_i$ for all $0\le i\le n-1$, $i\not\equiv 0\pmod k$.
\end{proof}

Rational monomials in $\mathcal G(n,q)$ were determined in \cite{Mor06} for $n=2$ and $q=p$. In general, we have the following

\begin{prop}\label{P4.3} For $e_0,\dots,e_{n-1}\in\Bbb Z$, $\x_0^{e_0}\cdots\x_{n-1}^{e_{n-1}}\in\mathcal G(n,q)$ if and only if $e_0+e_1q^{\frac 1n}+\cdots+e_{n-1}q^{\frac{n-1}n}\in\Bbb Z[q^{\frac 1n}]^\times$. Moreover, the mapping
\[
\begin{array}{ccccc}
\phi_n:&\Bbb Z[q^{\frac 1n}]^\times&\longrightarrow&\mathcal G(n,q)\cr
&e_0+e_1q^{\frac 1n}+\cdots+e_{n-1}q^{\frac{n-1}n}&\longmapsto&\x_0^{e_0}\cdots\x_{n-1}^{e_{n-1}},&e_0,\dots,e_{n-1}\in\Bbb Z
\end{array}
\]
is a group embedding.
\end{prop}

\begin{proof} In the first claim, the ``if'' part follows from the ring isomorphism \eqref{2.11}; the ``only if'' part follows from Lemma~\ref{L4.4} (i). The second claim follows from the ring isomorphism \eqref{2.11}.
\end{proof}

\noindent
{\bf Remark.} The ring $\Bbb Z[q^{\frac 1n}]$ is an {\em order} of the number field $K=\Bbb Q(q^{\frac 1n})$, i.e., a subring of $\frak o_K$ (the ring of integers of $K$) which is also a free $\Bbb Z$-module of rank $n$. By Dirichlet's unit theorem \cite{God03,Kle94},
\[
\Bbb Z[q^{\frac 1n}]^\times\cong\{\pm 1\}\times\Bbb Z^{\lfloor\frac n2\rfloor}.
\]
A $\Bbb Z$-basis of $\Bbb Z[q^{\frac 1n}]^\times/\{\pm 1\}$ is called a system of {\em fundamental units} of $\Bbb Z[q^{\frac 1n}]$. For algorithms for computing fundamental units of orders, see \cite{Buc-Pet89,Coh96}. For example, $1+\sqrt 2$ is a fundamental unit of $\Bbb Z[\sqrt 2]$; $\{-2+3^{3/6},\, -1-3^{1/6}+3^{2/6}-3^{3/6}+3^{5/6},\, 1-3^{1/6}-3^{2/6}-3^{3/6}+3^{5/6}\}$ is a system of fundamental units of $\Bbb Z[3^{1/6}]$ \cite[Table 1]{Buc-Pet89}.

\begin{lem}\label{L4.4}
If $f\in\mathcal G(n,q)$, then the following hold.
\begin{itemize}
  \item [(i)] $d_{\max}(f),\, d_{\min}(f)\in\Bbb Z[q^{\frac 1n}]^\times\cup\{0\}$.
  \item [(ii)] $d_{\max}(f)d_{\min}(f)\ge 0$.
  \item [(iii)] If $q=2$, then $\delta(f)\ne(0,0)$.
\end{itemize}
\end{lem}

\begin{proof} (i) Assume $d_{\max}(f)\ne 0$, say, $d_{\max}(f)>0$. Then by Lemma~\ref{L2.5},
\[
1=d_{\max}(\x_0)=d_{\max}(f^{(-1)}\circ f)=d_{\max}(f^{(-1)})d_{\max}(f).
\]
Hence $d_{\max}(f)\in\Bbb Z[q^{\frac 1n}]^\times$. In the same way, $d_{\min}(f)\in\Bbb Z[q^{\frac 1n}]^\times\cup\{0\}$.

\medskip
(ii) Assume to the contrary that $d_{\max}(f)d_{\min}(f)<0$. Without loss of generality, assume $d_{\max}(f)>0$ and $d_{\min}(f)<0$. Then by Lemma~\ref{L2.5}, 
\begin{gather*}
1=d_{\max}(f^{(-1)}\circ f)=d_{\max}(f^{(-1)})d_{\max}(f),\cr
1=d_{\min}(f^{(-1)}\circ f)=d_{\max}(f^{(-1)})d_{\min}(f),
\end{gather*}
which cannot be both true.

\medskip
(iii) Assume to the contrary that $\delta(f)=(0,0)$. Since $q=2$, we have $d_{\max}(f+1)<0$ and $d_{\min}(f+1)>0$. By (ii), $f+1\notin\mathcal G(n,2)$, which is a contradiction.
\end{proof}

Let $\mathcal M(n,q)$ denote the set of all rational monomials in $\mathcal G(n,q)$, that is,
\[
\mathcal M(n,q)=\{\x_0^{e_0}\cdots\x_{n-1}^{e_{n-1}}:e_0+e_1q^{\frac 1n}+\cdots+e_{n-1}q^{\frac{n-1}n}\in\Bbb Z[q^{\frac 1n}]^\times\}=\text{Im}\,\phi_n.
\]
There are several open questions about the fundamental structure of the group $\mathcal G(n,q)$.

\begin{ques}\label{Q4.5}\rm
Does $\mathcal G(1,q)\cup\mathcal M(n,q)$ generate $\mathcal G(n,q)$?
\end{ques}

\begin{ques}\label{Q4.6}\rm
Does $\bigl(\bigcup_{m\mid n,\,m<n}\mathcal G(m,q)\bigr)\cup\mathcal M(n,q)$ generate $\mathcal G(n,q)$? (By the embedding in Proposition~\ref{P4.2}, $\mathcal G(m,q)$ is treated as a subgroup of $\mathcal G(n,q)$ for $m\mid n$.) 
\end{ques}

\begin{ques}\label{Q4.7}\rm
We have $\mathcal G(1,q)\cap\mathcal M(n,q)=\langle\x_0^{-1}\rangle=\{\x_0,\x_0^{-1}\}$. When $q=2$, is $\langle\mathcal G(1,q)\cup\mathcal M(n,q)\rangle$ equal to the amalgamated product $\mathcal G(1,q)*_{\langle\x_0^{-1}\rangle}\mathcal M(n,q)$?
\end{ques}

\begin{ques}\label{Q4.8}\rm
Let $q=2$ and $n=2$. Is $\langle\x_0+1,\x_0\x_1\rangle$ the free product of $\langle\x_0+1\rangle$ and $\langle\x_0\x_1\rangle$?
\end{ques}

\begin{ques}\label{Q4.9}\rm 
Let $q=2$ and $n\ge 2$. Does $Q_n$ belong to $\langle\mathcal G(1,q)\cup\mathcal M(n,q)\rangle$? ($Q_n$ is the Dobbertin global $\mathcal P$-form given in \eqref{1.2}.)
\end{ques}

\noindent
{\bf Remark.}
\begin{itemize}
  \item [(i)] For general $q$, the answer to Question~\ref{Q4.7} is negative. For example, let $q=3$ and $n=2$. Then $-\x_0\in\mathcal G(1,3)\setminus\langle\x_0^{-1}\rangle$ commutes with $\x_0^2\x_1\in\mathcal M(2,3)$. It follows that $\langle\mathcal G(1,3)\cup\mathcal M(2,3)\rangle\ne \mathcal G(1,3)*_{\langle\x_0^{-1}\rangle}\mathcal M(2,3)$.
\smallskip 
  \item [(ii)] A positive answer to Question~\ref{Q4.7} implies a positive answer to Question~\ref{Q4.8}.
\smallskip
  \item [(iii)] When $q=2$ and $n=2$, there is a mysterious relation between $Q_2$ and $\x_0+1$ found by computer: $((\x_0+1)\circ Q_2)^{(3)}=\x_0$. Because of the this relation, we see that for $q=2$ and $n=2$, a positive answer to Question~\ref{4.7} also implies a negative answer to Question~\ref{Q4.9}.
\end{itemize}  

\medskip
We provide some evidence supporting a possible positive answer to Question~\ref{Q4.8}.

\begin{prop}\label{P4.10}
Assume 
\[
f=(\x_0\x_1)^{(e_1)}\circ(\x_0+1)\circ(\x_0\x_1)^{(e_2)}\circ(\x_0+1)\circ\cdots\circ(\x_0\x_1)^{(e_n)}\circ(\x_0+1)\in\mathcal G(2,2),
\]
where $0\ne e_i\in\Bbb Z$ for all $1\le i\le n$. If $f=\x_0$, then $n=2m\ge 6$ and $e_1+e_3+\cdots+e_{2m-1}=e_2+e_4+\cdots+e_{2m}=0$. 
\end{prop}

\begin{proof}
Since 
\[
(\x_0+1)\circ(\x_0\x_1)\circ(\x_0+1)=(\x_0+1)\circ(\x_0\x_1+\x_0+\x_1+1)=\x_0\x_1+\x_0+\x_1,
\]
we have 
\begin{equation}\label{4.3}
\delta\bigl((\x_0+1)\circ(\x_0\x_1)\circ(\x_0+1)\bigr)=(1+\sqrt 2,\,1).
\end{equation}
When $n=2m$, it follows from \eqref{4.3} and Lemma~\ref{L2.5} that
\begin{align}
\label{4.4}
d_{\max}(f)&=(1+\sqrt 2)^{e_1+\cdots+e_n},\\
\label{4.5}
d_{\min}(f)&=(1+\sqrt 2)^{e_1+e_3+\cdots+e_{2m-1}}.
\end{align}
We have $d_{\max}((\x_0\x_1)^{(e_1)}\circ(\x_0+1))=(1+\sqrt 2)^{e_1}$, and by Remark~\ref{R2.6}, $d_{\min}((\x_0\x_1)^{(e_1)}\circ(\x_0+1))=d_{\min}\bigl(((\x_0+1)(\x_1+1))^{(e_1)}\bigr)=0$. Hence for $n=2m+1$, we have 
\begin{align}
\label{4.6}
d_{\max}(f)&=(1+\sqrt 2)^{e_1+\cdots+e_n},\\
\label{4.7}
d_{\min}(f)&=0.
\end{align}
Now assume $f=\x_0$. By \eqref{4.4}, \eqref{4.5} and \eqref{4.7}, we must have $n=2m$ and $e_1+e_3+\cdots+e_{2m-1}=e_2+e_4+\cdots+e_{2m}=0$. It remains to show that $m\ge 3$. Assume to the contrary that $m\le 2$. We only have to consider the case $m=2$ since $m=1$ is clearly impossible. Then $e_1=-e_3$ and $e_2=-e_4$.
Thus
\[
(\x_0\x_1)^{(e_1)}\circ(\x_0+1)\circ(\x_0\x_1)^{(e_2)}\circ(\x_0+1)=(\x_0+1)\circ(\x_0\x_1)^{(e_2)}\circ(\x_0+1)\circ(\x_0\x_1)^{(e_1)},
\]
that is, $(\x_0+1)\circ(\x_0\x_1)^{(e_2)}\circ(\x_0+1)$ commutes with $(\x_0\x_1)^{(e_1)}$. By Lemma~\ref{L4.11}, $(\x_0+1)\circ(\x_0\x_1)^{(e_2)}\circ(\x_0+1)\in\mathcal M(2,2)$. However, by \eqref{4.3},
\[
d_{\max}\bigl((\x_0+1)\circ(\x_0\x_1)^{(e_2)}\circ(\x_0+1)\bigr)=(1+\sqrt 2)^{e_2}\ne 1=d_{\min}\bigl((\x_0+1)\circ(\x_0\x_1)^{(e_2)}\circ(\x_0+1)\bigr).
\]
Thus we have $(\x_0+1)\circ(\x_0\x_1)^{(e_2)}\circ(\x_0+1)\notin\mathcal M(2,2)$, which is a contradiction.
\end{proof}

We introduce some new notation to make the proof of the next result easier. We identity $\x_i$ with $\T^{q^{\frac in}}$; hence $\Bbb F_q(\x_0,\dots,\x_{n-1})=\Bbb F_q(\T,\T^{q^{\frac 1n}},\dots,\T^{q^{\frac{n-1}n}})$. For $g(\T)\in \Bbb F_q(\T,\T^{q^{\frac 1n}},\dots,\T^{q^{\frac{n-1}n}})$ and $e=e_0+e_1q^{\frac 1n}+\cdots+e_{n-1}q^{\frac{n-1}n}\in\Bbb Z[q^{\frac 1n}]$, where $e_i\in\Bbb Z$, define
\[
g^e=g(\T)^{e_0}g(\T^{q^{\frac 1n}})^{e_1}\cdots g(\T^{q^{\frac {n-1}n}})^{e_{n-1}}.
\]
Put
\[
\begin{split}
\Bbb Z[q^{\frac 1n}]_+\,&=\{a\in\Bbb Z[q^{\frac 1n}]:a>0\},\cr
\Bbb Z[q^{\frac 1n}]^\times_+\,&=\{a\in\Bbb Z[q^{\frac 1n}]^\times:a>0\},\cr
\Bbb Z_{\ge 0}[q^{\frac 1n}]\,&=\{b_0+b_1q^{\frac 1n}+\cdots+b_{n-1}q^{\frac{n-1}n}:b_i\in\Bbb Z,\ b_i\ge 0\}.
\end{split}
\]
If $g=1+c_1\T^{a_1}+\cdots+\cdot+c_k\T^{a_k}$, where $c_i\in\Bbb F_q$, $a_i\in\Bbb Z[q^{\frac 1n}]_+$, and $e\in\Bbb Z[q^{\frac 1n}]$, we have
\begin{equation}\label{4.8}
g^e=\sum_{b\in \Bbb Z_{\ge 0}[q^{\frac 1n}]}\binom eb(c_1\T^{a_1}+\cdots+c_k\T^{a_k})^b,
\end{equation}
where 
\[
\binom eb=\binom{e_0}{b_0}\cdots\binom{e_{n-1}}{b_{n-1}}
\]
for $e=e_0+e_1q^{\frac 1n}+\cdots+e_{n-1}q^{\frac{n-1}n}$ and $b=b_0+b_1q^{\frac 1n}+\cdots+b_{n-1}q^{\frac{n-1}n}$. The right side of \eqref{4.8} is a formal power series in $\T$ with exponents in $\Bbb Z_{\ge 0}[q^{\frac 1n}]a_1+\cdots+\Bbb Z_{\ge 0}[q^{\frac 1n}]a_k$.

\begin{lem}\label{L4.11}
Let $e=e_0+e_1q^{\frac 1n}+\cdots+e_{n-1}q^{\frac{n-1}n}\in \Bbb Z[q^{\frac 1n}]^\times_+\setminus\{1\}$, where $e_i\in\Bbb Z$. Assume that $f\in\Bbb F_q(\x_0,\dots\x_{n-1})\setminus\{0\}$ satisfies 
\begin{equation}\label{4.9}
f\circ(\x_0^{e_0}\cdots\x_{n-1}^{e_{n-1}})=(\x_0^{e_0}\cdots\x_{n-1}^{e_{n-1}})\circ f.
\end{equation}
Then $f=c\,\alpha$, where $\alpha\in\frak M(n,\Bbb F_q)$, $c\in\Bbb F_q\setminus\{0\}$ and $c^{e_0+\cdots+e_{n-1}-1}=1$.
\end{lem}

\begin{proof}
If $\beta\in \Bbb F_q(\x_0,\dots\x_{n-1})$ is a rational monomial, $\beta f$ also commutes with $\x_0^{e_0}\cdots\x_{n-1}^{e_{n-1}}$. Therefore we may assume $d_{\min}(f)=0$. 
Using the above notation, we can write
\[
f=c\biggl(1+\sum_{a\in\Bbb Z[q^{\frac 1n}]_+}c_a\T^a\biggr),
\]
where $c\in \Bbb F_q\setminus\{0\}$, $c_a\in\Bbb F_q$ and $\{a:c_a\ne 0\}\subset\Bbb Z_{\ge0}[q^{\frac 1n}]a_1+\cdots+\Bbb Z_{\ge0}[q^{\frac 1n}]a_k$ for some $a_1,\dots,a_k\in\Bbb Z[q^{\frac 1n}]_+$. 
Equation~\eqref{4.9} clearly implies $c^{e_0+\cdots+e_{n-1}-1}=1$. Thus we may assume $c=1$.
If $f=1$, we are done. So assume to the contrary that $f\ne 1$.
Equation~\eqref{4.9} becomes $f(\T^e)=f(\T)^e$, i.e.,
\begin{equation}\label{4.10}
1+\sum_{a\in \Bbb Z[q^{\frac 1n}]_+}c_a\T^{ae}=\Bigl(1+\sum_{a\in\Bbb Z[q^{\frac 1n}]_+}c_a\T^a\Bigr)^e.
\end{equation}
Let $a_0\in\Bbb Z[q^{\frac 1n}]_+$ be minimum such that $c_{a_0}\ne 0$. Then the second lowest term in the left side of \eqref{4.10} is $c_{a_0}\T^{a_0e}$. Since $e\in \Bbb Z[q^{\frac 1n}]^\times$, we have $e_0\not\equiv 0\pmod p$. (For this fact, confer the proof of Proposition~\ref{P4.14}.) Thus the second lowest term in the right side of \eqref{4.10} is 
\[
\binom e1 c_{a_0}\T^{a_0}=e_0c_{a_0}\T^{a_0}.
\]
Since $e\ne 1$, we have a contradiction.
\end{proof}

\begin{cor}\label{C4.12}
Let $\epsilon_i=e_{i,0}+e_{i,1}q^{\frac 1n}+\cdots+e_{i,n-1}q^{\frac{n-1}n}$, $1\le i\le\lfloor\frac n2\rfloor$, be a system of fundamental units of $\Bbb Z[q^{\frac 1n}]$. Let $d=\text{\rm gcd}\bigl(\{e_{i,0}+\cdots+e_{i,n-1}-1:1\le i\le\lfloor\frac n2\rfloor\}\cup\{q-1\}\bigr)$ and put $\mu_d=\{c\in\Bbb F_q\setminus\{0\}:c^d=1\}$. Then
$\{c\,\x_0:c\in\mu_d\}\times\mathcal M(n,q)$ is a maximal abelian subgroup of $\mathcal G(n,q)$.
\end{cor}

Let 
\[
\mathcal H(n,q)=\{f\in\mathcal G(n,q):d_{\max}(f)d_{\min}(f)\ne 0\}.
\]
By Lemma~\ref{L2.5}, $\mathcal H(n,q)$ is a subgroup of $\mathcal G(n,q)$ and $\mathcal H(n,q)$ has a normal subgroup $\mathcal H_0(n,q)=\{f\in\mathcal G(n,q):\delta(f)=(1,1)\}$. We will see that $[\mathcal G(n,q):\mathcal H(n,q)]<\infty$ and we will determine a system of representatives of the left cosets of $\mathcal H(n,q)$ in $\mathcal G(n,q)$. We will also determine the structure of $\mathcal H(n,q)/\mathcal H_0(n,q)$.

Let $\frak A$ be a set of ordered pairs $(a,b)\in \Bbb F_q^2$, $ab\ne 0$, $a\ne b$, that represent each two element subset of $\Bbb F_q\setminus\{0\}$ precisely once.

\begin{prop}\label{P4.13}
A system of representatives of the left cosets of $\mathcal H(n,q)$ in $\mathcal G(n,q)$ is given by
\[
\mathcal L=\{\x_0+a:a\in\Bbb F_q\}\cup\Bigl\{\frac 1{\x_0+b}:b\in\Bbb F_q\setminus\{0\}\Bigr\}\cup\Bigl\{\frac{a\x_0+b}{\x_0+1}:(a,b)\in\frak A\Bigr\}.
\]
In particular, $[\mathcal G(n,q):\mathcal H(n,q)]=\frac 12q(q+1)$.
\end{prop}

\begin{proof}
$1^\circ$ We first show that $\mathcal G(n,q)=\bigcup_{\alpha\in\mathcal L}\alpha\circ\mathcal H(n,q)$. Let $f\in\mathcal G(n,q)$. If $f\in\mathcal H(n,q)$, then $f\in\x_0\circ\mathcal H(n,q)$. So we assume $f\notin\mathcal H(n,q)$.

\medskip
{\bf Case 1.} Assume that $d_{\max}(f)=0$ and $d_{\min}(f)<0$. The following table illustrates the coefficients of the numerator and the denominator of $f$ in the decreasing order of the term degree, where $abc\ne 0$.
\[
\begin{tabular}{l|l}
coef's of the numerator of $f$ & $a\,\cdots\cdots\, b$\\ \hline
coef's of the denominator of $f$ & $1\,\cdots\, c$
\end{tabular}
\]
In this case we have $(\x_0-a)\circ f\in\mathcal H(n,q)$ and hence $f\in(\x_0+a)\circ\mathcal H(n,q)$.

\medskip
{\bf Case 2.} Assume that $d_{\max}(f)=0$ and $d_{\min}(f)>0$; see the following table, where $abc\ne 0$.
\[
\begin{tabular}{l|l}
coef's of the numerator of $f$ &$1\,\cdots\, c$ \\ \hline
coef's of the denominator of $f$ & $a\,\cdots\cdots\, b$
\end{tabular}
\]
We have $(\x_0-a)\circ\frac 1{\x_0}\circ f\in\mathcal H(n,q)$. Hence $f\in(\frac 1{\x_0})^{(-1)}\circ(\x_0-a)^{(-1)}\circ\mathcal H(n,q)=\frac 1{\x_0+a}\circ\mathcal H(n,q)$.

\medskip
{\bf Case 3.} Assume that $d_{\min}(f)=0$ and $d_{\max}(f)\ne 0$. The proof is similar to Cases 1 and 2.

\medskip
{\bf Case 4.} Assume that $d_{\max}(f)=d_{\min}(f)=0$; see the following table, where $abc\ne 0$ and $ac-b\ne 0$. (Note. If $ac-b=0$, then $d_{\max}(f-a)d_{\min}(f-a)<0$, which is impossible by Lemma~\ref{L4.4} (ii).)
\[
\begin{tabular}{l|l}
coef's of the numerator of $f$ & $a\,\cdots\cdots\, b$\\ \hline
coef's of the denominator of $f$ & $1\,\cdots\cdots\, c$
\end{tabular}
\]
Then we have
\[
\Bigl(\x_0-\frac c{b-ac}\Bigr)\circ\frac 1{\x_0}\circ(x_0-a)\circ f\in\mathcal H(n,q).
\]
Therefore
\begin{align*}
f\,&\in \Bigl[\Bigl(\x_0-\frac c{b-ac}\Bigr)\circ\frac 1{\x_0}\circ(x_0-a)\Bigr]^{(-1)}\circ\mathcal H(n,q)\cr
&=(\x_0+a)\circ\frac 1{\x_0}\circ\Bigl(\x_0+\frac c{b-ac}\Bigr)\circ\mathcal H(n,q)\cr
&=\frac{a\x_0+\frac b{b-ac}}{\x_0+\frac c{b-ac}}\circ\mathcal H(n,q)\cr
&=\frac{a\x_0+b'}{\x_0+c'}\circ\mathcal H(n,q)\kern 2cm (b'c'\ne 0)\cr
&=\frac{a\x_0+b'}{\x_0+c'}\circ(c'\x_0)\circ\frac{\x_0}{c'}\circ\mathcal H(n,q)\cr
&=\frac{a\x_0+b''}{\x_0+1}\circ\mathcal H(n,q)\kern 2cm (b''\ne 0)\cr
&=\frac{b''\x_0+a}{\x_0+1}\circ\mathcal H(n,q)\kern 2cm (\,\frac{b''\x_0+a}{\x_0+1}=\frac{a\x_0+b''}{\x_0+1}\circ\frac 1{\x_0}\,),
\end{align*}
where either $(a,b'')$ or $(b'',a)$ belongs $\frak A$.

\medskip
$2^\circ$ For any $\alpha,\beta\in\mathcal L$ with $\alpha\ne \beta$, we show that $\alpha^{(-1)}\circ\beta\notin\mathcal H(n,q)$. Note that $\mathcal L\subset \mathcal G(1,q)$ and under the isomorphism $\mathcal G(1,q)\cong\text{PGL}(2,q)$, $\mathcal L$ corresponds to
\[
\frak L=\Bigl\{\left[\begin{matrix} 1&a\cr 0&1\end{matrix}\right]:a\in\Bbb F_q\Bigr\}\cup\Bigl\{\left[\begin{matrix} 0&1\cr 1&b\end{matrix}\right]:0\ne b\in\Bbb F_q\Bigr\}\cup\Bigl\{\left[\begin{matrix} a&b\cr 1&1\end{matrix}\right]:(a,b)\in\frak A\Bigr\}.
\]
It suffices to show that for all $A,B\in\frak L$ with $A\ne B$, $A^{-1}B$ is not of the form
\[
\left[\begin{matrix} *&0\cr 0&*\end{matrix}\right]\quad \text{or}\quad \left[\begin{matrix} 0&*\cr *&0\end{matrix}\right].
\]

{\bf Case 1.} Assume $A=\left[\begin{smallmatrix}1&a_1\cr 0&1\end{smallmatrix}\right]$, $B=\left[\begin{smallmatrix}1&a_2\cr 0&1\end{smallmatrix}\right]$, $a_1\ne a_2$. Then $A^{-1}B=\left[\begin{smallmatrix}1&-a_1+a_2\cr 0&1\end{smallmatrix}\right]$.

\medskip
{\bf Case 2.} Assume $A=\left[\begin{smallmatrix}1&a\cr 0&1\end{smallmatrix}\right]$, $B=\left[\begin{smallmatrix}0&1\cr 1&b\end{smallmatrix}\right]$, $b\ne 0$. Then $A^{-1}B=\left[\begin{smallmatrix}*&*\cr 1&b\end{smallmatrix}\right]$.

\medskip
{\bf Case 3.} Assume $A=\left[\begin{smallmatrix}1&a_1\cr 0&1\end{smallmatrix}\right]$, $B=\left[\begin{smallmatrix}a_2&b_2\cr 1&1\end{smallmatrix}\right]$, $(a,b)\in\frak A$. Then $A^{-1}B=\left[\begin{smallmatrix}*&*\cr 1&1\end{smallmatrix}\right]$.

\medskip
{\bf Case 4.} Assume $A=\left[\begin{smallmatrix}0&1\cr 1&b_1\end{smallmatrix}\right]$, $B=\left[\begin{smallmatrix}0&1\cr 1&b_2\end{smallmatrix}\right]$, $b_1b_2\ne 0$, $b_1\ne b_2$. Then $A^{-1}B=\left[\begin{smallmatrix}1&-b_1+b_2\cr 0&1\end{smallmatrix}\right]$.

\medskip
{\bf Case 5.} Assume $A=\left[\begin{smallmatrix}0&1\cr 1&b_1\end{smallmatrix}\right]$, $B=\left[\begin{smallmatrix}a_2&b_2\cr 1&1\end{smallmatrix}\right]$, $(a_2,b_2)\in\frak A$. Then $A^{-1}B=\left[\begin{smallmatrix}*&*\cr a_2&b_2\end{smallmatrix}\right]$.

\medskip
{\bf Case 6.} Assume $A=\left[\begin{smallmatrix}a_1&b_1\cr 1&1\end{smallmatrix}\right]$, $B=\left[\begin{smallmatrix}a_2&b_2\cr 1&1\end{smallmatrix}\right]$, $(a_1,b_1),(a_2,b_2)\in\frak A$, $(a_1,b_1)\ne(a_2,b_2)$. Then
\[
A^{-1}B=\frac 1{a_1-b_1}\left[\begin{matrix}a_2-b_1&b_2-b_1\cr -a_2+a_1&-b_2+a_1\end{matrix}\right],
\]
which is not of the form $\left[\begin{smallmatrix}*&0\cr 0&*\end{smallmatrix}\right]$ or $\left[\begin{smallmatrix}0&*\cr *&0\end{smallmatrix}\right]$ since $(a_1,b_1)\ne(a_2,b_2),(b_2,a_2)$.
\end{proof}

Let $\Delta=\{(a,b):a,b\in\Bbb Z[q^{\frac 1n}]^\times,\ ab>0\}$. For $(a_1,b_1),(a_2,b_2)\in\Delta$, define 
\[
(a_1,b_1)(a_2,b_2)=
\begin{cases}
(a_1a_2,b_1b_2)&\text{if}\ a_2>0,\ b_2>0,\cr
(b_1a_2,a_1b_2)&\text{if}\ a_2<0,\ b_2<0.
\end{cases}
\]
Then $\Delta$ is a group. In fact,  let $\tau$ be the order $2$ automorphism of $(\Bbb Z[q^{\frac 1n}]^\times_+)^2$ defined by $\tau(a,b)=(b,a)$. Then $\Delta\cong\langle\tau\rangle\ltimes(\Bbb Z[q^{\frac 1n}]^\times_+)^2$. (Note. By Dirichlet's unit theorem, $\Bbb Z[q^{\frac 1n}]^\times_+\cong\Bbb Z^{\lfloor\frac n2\rfloor}$.)

\begin{prop}\label{P4.14}
The mapping $\delta:\mathcal H(n,q)\to\Delta$ defined in \eqref{2.6} is an onto group homomorphism with $\ker\delta=\mathcal H_0(n,q)$.
\end{prop}

\begin{proof}
That $\delta$ is a group homomorphism follows from Lemma~\ref{L2.5}. To prove that $\delta$ is onto, it suffices to show that fro each $e\in\Bbb Z[q^{\frac 1n}]^\times_+$, there exists $f\in\mathcal H(n,q)$ such that $\delta(f)=(e,1)$. (Note that $\Delta$ is generated by $\Bbb Z[q^{\frac 1n}]^\times_+\times\{1\}$ and $(-1,-1)$, where $(-1,-1)=\delta(\frac 1{\x_0})$.) Write
\begin{equation}\label{4.11}
e=e_0+e_1q^{\frac 1n}+\cdots+e_{n-1}q^{\frac{n-1}n},\quad e_i\in\Bbb Z.
\end{equation}
Let $p=\text{char}\,\Bbb F_q$ and $\nu_p$ the $p$-adic valuation of $\Bbb Q(q^{\frac 1n})$. Then $\nu_p(e)=0$ since $e$ is a unit. Thus \eqref{4.11} gives $\nu_p(e_0)=0$, i.e., $p\nmid e_0$. Let $I=\{0\le i\le n-1:e_i>0\}$ and $J=\{0\le j\le n-1:e_j<0\}$. Then
\begin{align*}
\mathcal G(n,q)\,&\ni (\x_0-1)\circ(\x_0^{e_0}\cdots\x_{n-1}^{e_{n-1}})\circ(\x_0+1)\cr
&=(\x_0-1)\circ\frac{\prod_{i\in I}\x_i^{e_i}}{\prod_{j\in J}\x_j^{-e_j}}\circ(\x_0+1)\cr
&=\frac{\prod_{i\in I}\x_i^{e_i}-\prod_{j\in J}\x_j^{-e_j}}{\prod_{j\in J}\x_j^{-e_j}}\circ(\x_0+1)\cr
&=\frac{\prod_{i\in I}(\x_i+1)^{e_i}-\prod_{j\in J}(\x_j+1)^{-e_j}}{\prod_{j\in J}(\x_j+1)^{-e_j}}\cr
&=\frac{\prod_{i\in I}\x_i^{e_i}+\cdots+e_0\x_0}{\prod_{j\in J}\x_j^{-e_j}+\cdots+1}.
\end{align*}
The above element has $d_{\max}=e$ and $d_{\min}=1$.
\end{proof}


\end{document}